\documentclass[12pt,leqno,final]{article}
\usepackage{amsmath, verbatim, color}
\usepackage{mathrsfs}
\usepackage{dsfont}
\usepackage[a4paper,margin=0.91in]{geometry}
\usepackage[notcite,notref]{showkeys}
\usepackage{amssymb}
\usepackage{dsfont,mathrsfs}
\usepackage{amsmath, amsthm, amsfonts, amssymb, color}
\usepackage{mathrsfs}
\usepackage{color}
\usepackage{stmaryrd}
\makeatletter
\newcommand{\Spvek}[2][r]{%
  \gdef\@VORNE{1}
  \left(\hskip-\arraycolsep%
    \begin{array}{#1}\vekSp@lten{#2}\end{array}%
  \hskip-\arraycolsep\right)}

\def\vekSp@lten#1{\xvekSp@lten#1;vekL@stLine;}
\def\vekL@stLine{vekL@stLine}
\def\xvekSp@lten#1;{\def\temp{#1}%
  \ifx\temp\vekL@stLine
  \else
    \ifnum\@VORNE=1\gdef\@VORNE{0}
    \else\@arraycr\fi%
    #1%
    \expandafter\xvekSp@lten
  \fi}
\makeatother

\newtheorem{thm}{Theorem}[section]

\newtheorem{lem}[thm]{Lemma}

\newtheorem{exa}[thm]{Example}
\newtheorem{rem}[thm]{Remark}
\theoremstyle{definition}

\newcommand{\scr}[1]{\mathscr #1}
\definecolor{wco}{rgb}{0.5,0.2,0.3}
\newcommand{\I}{\mathds{1}}
\numberwithin{equation}{section} \theoremstyle{remark}

\newcommand{\ua}{\uparrow}

\title{{\bf   Well-Posedness for McKean-Vlasov SDEs Driven by Multiplicative Stable Noises}
}
\author{
{\bf     Chang-Song Deng $^{a)}$, Xing Huang $^{b)}$,    }\\
\footnotesize{  a)School of Mathematics and Statistics, Wuhan University, Wuhan 430072, China}\\
\footnotesize{ dengcs@whu.edu.cn }\\
\footnotesize{  b)Center for Applied Mathematics, Tianjin University, Tianjin 300072, China}\\
\footnotesize{  xinghuang@tju.edu.cn}}
\begin{document}
\allowdisplaybreaks
\def\R{\mathbb R}  \def\ff{\frac} \def\ss{\sqrt} \def\B{\mathbf
B} \def\W{\mathbb W}
\def\N{\mathbb N} \def\kk{\kappa} \def\m{{\bf m}}
\def\ee{\varepsilon}\def\ddd{D^*}
\def\dd{\delta} \def\DD{\Delta} \def\vv{\varepsilon} \def\rr{\rho}
\def\<{\langle} \def\>{\rangle} \def\GG{\Gamma} \def\gg{\gamma}
  \def\nn{\nabla} \def\pp{\partial} \def\E{\mathbb E}
\def\d{\text{\rm{d}}} \def\bb{\beta} \def\aa{\alpha} \def\D{\scr D}
  \def\si{\sigma} \def\ess{\text{\rm{ess}}}
\def\beg{\begin} \def\beq{\begin{equation}}  \def\F{\scr F}
\def\Ric{\text{\rm{Ric}}} \def\Hess{\text{\rm{Hess}}}
\def\e{\text{\rm{e}}}
\def\iup{\text{\rm{i}}}
\def\ua{\underline a} \def\OO{\Omega}  \def\oo{\omega}
 \def\tt{\tilde} \def\Ric{\text{\rm{Ric}}}
\def\cut{\text{\rm{cut}}} \def\P{\mathbb P} \def\ifn{I_n(f^{\bigotimes n})}
\def\C{\scr C}      \def\aaa{\mathbf{r}}     \def\r{r}
\def\gap{\text{\rm{gap}}} \def\prr{\pi_{{\bf m},\varrho}}  \def\r{\mathbf r}
\def\Z{\mathbb Z} \def\vrr{\varrho} \def\ll{\lambda}
\def\L{\scr L}\def\Tt{\tt} \def\TT{\tt}
\def\i{{\rm in}}\def\Sect{{\rm Sect}}  \def\H{\mathbb H}
\def\M{\scr M}\def\Q{\mathbb Q} \def\texto{\text{o}} \def\LL{\Lambda}
\def\Rank{{\rm Rank}} \def\B{\scr B} \def\i{{\rm i}} \def\HR{\hat{\R}^d}
\def\to{\rightarrow}\def\l{\ell}\def\iint{\int}
\def\EE{\scr E}\def\Cut{{\rm Cut}}
\def\A{\scr A} \def\Lip{{\rm Lip}}
\def\BB{\scr B}\def\Ent{{\rm Ent}}\def\L{\scr L}
\def\R{\mathbb R}  \def\ff{\frac} \def\ss{\sqrt} \def\B{\mathbf
B}
\def\N{\mathbb N} \def\kk{\kappa} \def\m{{\bf m}}
\def\dd{\delta} \def\DD{\Delta} \def\vv{\varepsilon} \def\rr{\rho}
\def\<{\langle} \def\>{\rangle} \def\GG{\Gamma} \def\gg{\gamma}
  \def\nn{\nabla} \def\pp{\partial} \def\E{\mathbb E}
\def\d{\text{\rm{d}}} \def\bb{\beta} \def\aa{\alpha} \def\D{\scr D}
  \def\si{\sigma} \def\ess{\text{\rm{ess}}}
\def\beg{\begin} \def\beq{\begin{equation}}  \def\F{\scr F}
\def\Ric{\text{\rm{Ric}}} \def\Hess{\text{\rm{Hess}}}
\def\ua{\underline a} \def\OO{\Omega}  \def\oo{\omega}
 \def\tt{\tilde} \def\Ric{\text{\rm{Ric}}}
\def\cut{\text{\rm{cut}}} \def\P{\mathbb P} \def\ifn{I_n(f^{\bigotimes n})}
\def\C{\scr C}      \def\aaa{\mathbf{r}}     \def\r{r}
\def\gap{\text{\rm{gap}}} \def\prr{\pi_{{\bf m},\varrho}}  \def\r{\mathbf r}
\def\Z{\mathbb Z} \def\vrr{\varrho} \def\ll{\lambda}
\def\L{\scr L}\def\Tt{\tt} \def\TT{\tt}\def\II{\mathbb I}
\def\i{{\rm in}}\def\Sect{{\rm Sect}}  \def\H{\mathbb H}
\def\M{\scr M}\def\Q{\mathbb Q} \def\texto{\text{o}} \def\LL{\Lambda}
\def\Rank{{\rm Rank}} \def\B{\scr B} \def\i{{\rm i}} \def\HR{\hat{\R}^d}
\def\to{\rightarrow}\def\l{\ell}
\def\8{\infty}\def\U{\scr U}

\maketitle

\begin{abstract}
We establish the well-posedness for a class of McKean-Vlasov SDEs driven by symmetric
$\alpha$-stable L\'{e}vy process ($1/2<\alpha\leq1$), where the drift coefficient is H\"{o}lder
continuous in space variable, while the noise coefficient is Lipscitz continuous in space variable,
and both of them satisfy the Lipschitz condition in distribution variable with respect to
Wasserstein distance. If the drift coefficient does not depend on distribution variable,
our methodology developed in this paper applies to the case $\alpha\in(0,1]$. The main tool
relies on heat kernel estimates for (distribution independent) stable SDEs and
Banach's fixed point theorem.
\end{abstract} \noindent

 \vskip 2cm

\section{Introduction}

It is well-known that many complex physical, biological, and other scientific phenomena can be
modeled by interacting particle systems, which attract much attention in recent years due to their
importance both in theory and in applications. When the number of particles goes to infinity,
the equation for one single particle in the mean field interacting particle system tends to the so-called McKean-Vlasov SDE,
which was first introduced by McKean in \cite{McKean}. This is related to the propagation of chaos, see for instance \cite{SZ}.
As a fundamental issue in the study of McKean-Vlasov SDEs, the well-posedness has been intensively investigated
for Gaussian noise case, see \cite{CF,HW22,23R,RZ, W21a,ZG} and references therein for more details.

As far as we know, however, the results concerning the well-posedness for McKean-Vlasov SDEs with jump noises are
still quite limited. The well-posedness is established in \cite{JMW} for L\'{e}vy-driven McKean-Vlasov
SDEs without drift. In \cite{HaW}, the authors consider strong well-posedness for density dependent SDEs with additive $\alpha$-stable noise ($1<\alpha<2$), where the drift is assumed to be $C_b^\beta$ with $\beta\in(1-\alpha/2,1)$ in space variable, and Lipschitz continuous in distribution variable with
 respect to the $L^\theta$-Wasserstein distance ($1<\theta<\alpha$).
 In \cite{NVS}, the authors prove the well-posedness for stable McKean-Vlasov SDEs under the assumption that the coefficients have
 bounded and H\"{o}lder continuous flat derivatives (also called linear functional derivatives); in the supercritical case,
 i.e.\ the stability index $\alpha<1$, it is necessary to require $\alpha>2/3$ (see \cite[Theorem 2.2]{NVS}).
In the very recent work \cite{DH23}, we establish the well-posedness for McKean-Vlasov SDEs driven by $\alpha$-stable noise ($1<\alpha<2$), where the noise coefficient depends only on time and distribution variables.

As a continuation of \cite{DH23}, in this paper, we consider the following stable McKean-Vlasov SDEs
with stable index $\alpha\in(0,1]$:
\begin{align}\label{E1}
\d X_t=b_t(X_t,\L_{X_t})\,\d t+\sigma_t(X_t,\L_{X_t})\,\d L_t,\quad t\in[0,T],
\end{align}
where $T>0$ is a fixed constant, $(L_t)_{t \ge 0}$ is a $d$-dimensional rotationally invariant
$\alpha$-stable L\'evy process with infinitesimal generator $-\frac12(-\triangle)^{\alpha/2}$, $\L_{X_t}$ is the law of $X_t$, and for the space
$\scr P$ of all probability measures on $\R^d$ equipped with the weak topology,
$$
    b:[0,T]\times\R^d\times\scr P\rightarrow\R^d,\quad \sigma:[0,T]\times\R^d\times\scr P\rightarrow\R^d\otimes\R^d
$$
are measurable.

For $\kappa\in(0,1]$, let $$\scr P_\kappa:=\big\{\gamma\in \scr P:\ \gamma(|\cdot|^\kappa)<\infty\big\},$$
which is a Polish space under the $L^\kappa$-Wasserstein distance
$$\W_\kappa(\gamma,\tilde{\gamma}):= \inf_{\pi\in \C(\gamma,\tilde{\gamma})} \int_{\R^{d}\times\R^{d}} |x-y|^\kappa \,
\pi(\d x,\d y),\quad \gamma,\tilde{\gamma}\in \scr P_\kappa, $$
where $\C(\gamma,\tilde{\gamma})$ is the set of all couplings of $\gamma$ and $\tilde{\gamma}$.
By \cite[Theorem 5.10]{Chen04},  the following adjoint formula holds:
$$\W_\kappa(\gamma,\tilde{\gamma})=\sup_{[f]_\kappa\leq 1}|\gamma(f)-\tilde{\gamma}(f)|,\quad\gamma,\tilde{\gamma}\in \scr P_\kappa,$$
where $[f]_\kappa$ denotes the H\"{o}lder seminorm (of exponent $\kappa$) of $f:\R^d\rightarrow\R$ defined by $[f]_\kappa:=\sup_{x\neq y}\frac{|f(x)-f(y)|}{|x-y|^\kappa}$.

  To derive the well-posedness for \eqref{E1}, we make the following assumptions.
\beg{enumerate}
\item[$(A)$] $\alpha\in(1/2,1]$. There exist   $\beta\in(0,1)$ satisfying $2\beta+\alpha>2$, $K>1$ and $\eta\in(0,\alpha)$ with $\alpha+\eta>1$ such that for all $t\in[0,T]$, $x,y\in\R^d$ and $\gamma,\tilde{\gamma}\in\scr P_\eta$,
\begin{align}\label{b-b}
|b_t(x,\gamma)-b_t(y,\tilde{\gamma})|\leq K\left(\W_{\eta}(\gamma,\tilde{\gamma})+\big\{|x-y|^\beta\vee|x-y|\big\}\right),
\end{align}
and
\begin{equation}\label{s-s}
\left\{\begin{aligned}
    &|b_t(0,\delta_0)|\leq K,
    \\
    &\|\sigma_t(x,\gamma)-\sigma_t(y,\tilde{\gamma})\|\leq K\big(|x-y|+\W_{\eta}(\gamma,\tilde{\gamma}
)\big),\\
    &K^{-1}I\leq (\sigma_t\sigma^\ast_t)(x,\gamma)\le KI.
\end{aligned}\right.
\end{equation}
\end{enumerate}
\begin{enumerate}
\item[$(A')$] $\alpha\in(0,1]$ and $b_t(x,\gamma)=b_t(x)$ does not depend on $\gamma$. There exist   $\beta\in(0,1)$ satisfying $2\beta+\alpha>2$, $K>1$ and $\eta\in(0,\alpha)$ such that for all $t\in[0,T]$, $x,y\in\R^d$ and $\gamma,\tilde{\gamma}\in\scr P_\eta$,
    $$
      |b_t(x)-b_t(y)|\leq K\big\{|x-y|^\beta\vee|x-y|\big\}
    $$
    and \eqref{s-s} hold.
\end{enumerate}

Denote by $C([0,T];\scr P_k)$ the set of all
continuous maps from $[0,T]$ to $\scr P_k$ under the metric $\W_k$. Throughout the paper the constant $C$
denotes positive constant which may depend on $T,d,\alpha,\beta,\eta,K$; its value may change,
without further notice, from line to line.

Our main result is the following theorem:
\begin{thm}\label{EUS} Assume $(A)$ or $(A')$.
Then \eqref{E1} is strongly/weakly well-posed in $\scr P_{\eta}$, and the solution satisfies $\L_{X_\cdot}\in C([0,T];\scr P_{\eta})$ and
$$
    \E\left[\sup_{t\in[0,T]}|X_t|^{\eta}\right]<C\left(
    1+\E\big[|X_0|^{\eta}\big]\right).
$$

\end{thm}

\begin{rem}
    If $\alpha<1$, to ensure the well-posedness, it is required in \cite[Theorem 2.2]{NVS}
    that $\alpha>2/3$. In Theorem \ref{EUS}, we can handle the case $\alpha\in(1/2,1)$, and even $\alpha\in(0,1)$ when
    $b_t(x,\gamma)$ does not depend on $\gamma$.
\end{rem}

\begin{rem}
    Let $\alpha\in(1/2,1)$. By Example \ref{conter} in the Appendix, the condition $\alpha+\eta>1$ in $(A)$
    is necessary in the sense that if $\alpha+\eta=1$, then we cannot expect the uniqueness for the solution
    to \eqref{E1}.
\end{rem}

The remainder of the paper is organized as follows: In Section \ref{prepa}, we make some preparations and the proof
of Theorem \ref{EUS} is presented in Section \ref{pf}. A counterexample is provided in the Appendix for
non-uniqueness of solutions to stable McKean-Vlasov SDEs

\section{Some preparations}\label{prepa}

Let $\gamma\in\scr P_\eta$ and $\mu\in C([0,T];\scr P_\eta)$, where $\eta\in(0,\alpha)$. Consider the following (distribution independent)
SDE with initial distribution $\L_{X_{s,s}^{\gg,\mu}}=\gg$:
 \beq\label{ED}
     \d X_{s,t}^{\gg,\mu}= b_t(X_{s,t}^{\gg,\mu}, \mu_t)\,\d t+\sigma_t(X_{s,t}^{\gg,\mu},\mu_t)\,\d L_t,\quad 0\leq s\leq t\leq T.
 \end{equation}
By \cite[Theorem 1.1]{CZZ21} and a standard localization argument, \eqref{ED} has a unique strong solution
under $(A)$ or $(A')$. For simplicity, we denote $X_{t}^{\gg,\mu}=X_{0,t}^{\gg,\mu}$. Moreover, if
$\gamma=\delta_x$ is the Dirac measure concentrated at $x\in\R^d$, we write $X_{s,t}^{x,\mu}=X_{s,t}^{\delta_x,\mu}$.

By \cite{MZ}, $\L_{X_{s,t}^{x,\mu}}$ is absolutely continuous with respect to
the Lebesgue measure, and we denote by $p_{s,t}^{\mu}(x,\cdot)$ the corresponding
density function. Denote by $P_{s,t}^{\mu}$ the inhomogeneous Markov semigroup associated
with $X_{s,t}^{x,\mu}$, i.e.\ for $f\in \scr B_b(\R^d)$,
\begin{align*}
    P_{s,t}^{\mu}f(x)&=\E f(X_{s,t}^{x,\mu})=\int_{\R^d}p_{s,t}^{\mu}(x,y)f(y)\,\d y.
 \end{align*}
Here and in the sequel, $\scr B_b(\R^d)$ denotes the set of all bounded measurable
functions on $\R^d$. As before, write $p_{t}^{\mu}(x,\cdot)=p_{0,t}^{\mu}(x,\cdot)$ and $P_{t}^{\mu}=P_{0,t}^{\mu}$
for $t\in[0,T]$ and $\mu\in C([0,T];\scr P_\eta)$. We will use the following notation
for $\mu\in C([0,T];\scr P_\eta)$
\beg{equation}\label{L34}
\A^{\mu}_t f(\cdot):=\int_{\mathbb{R}^{d}\backslash\{0\}}\big[f(\cdot+\sigma_t(\cdot,\mu_t)y)-f(\cdot)
-\langle \sigma_t(\cdot,\mu_t)y,\nabla f(\cdot)\rangle\mathds{1}_{\{ |y|\leq 1\}} \big]\,\Pi(\d y),
\end{equation}
where
$$
    \Pi(\d y):=\frac{\alpha\Gamma(\frac{d+\alpha}{2})}
    {2^{2-\alpha}\pi^{d/2}\Gamma(1-\frac\alpha2)}\,\frac{\d y}{|y|^{d+\alpha}}
$$
is the L\'{e}vy measure of $L_t$.

\beg{lem}\label{L2}  Assume $(A)$ or $(A')$.  For any $0\leq r<t\le T$, $
\mu^1,\mu^2 \in C([0,T];\scr P_{\eta})$, and $f\in\scr B_b(\R^d)$ with $[f]_\eta\leq 1$,
\begin{equation}\label{g0'}
|\nabla P_{r,t}^{\mu^2}f|\leq C(t-r)^{-\frac1\alpha+\frac\eta\alpha},
\end{equation}
\begin{equation}\label{g2'}
|(\scr A_r^{\mu^1}-\scr A_r^{\mu^2})P_{r,t}^{\mu^2}f|\leq C(t-r)^{-1+\frac{\eta}{\alpha}}\W_{\eta}(\mu_r^1,\mu_r^2).
\end{equation}
\end{lem}
\begin{proof}
(i) Denote by $p^\alpha_t$ the density function of $L_t$ (with respect to the Lebesgue measure).
It follows from \cite[Theorem 1.1\,(iii) and (i)]{MZ} that for any $0\leq r<t\le T$,
$$
    |\nabla p_{r,t}^{\mu^2}(\cdot, y)(x)|\leq C(t-r)^{-1/\alpha}p^\alpha_{t-r}\left(\theta_{r,t}(x)-y\right),
$$
where $\theta_{r,t}(x)$ denotes the flow associated to the drift in \eqref{ED}, i.e.
$$
    \left\{
    \begin{aligned}
        &\displaystyle\frac{\partial}{\partial t}\,\theta_{r,t}(x)=b\left(\theta_{r,t}(x)\right),\\
        &\theta_{r,r}(x)=x.
    \end{aligned}
    \right.
$$
Noting that
$$
    \nabla_x\underbrace{\int_{\R^d}p_{r,t}^{\mu^2}(x,y)\,\d y}_{=1}=0,
$$
we obtain for all $f\in\scr B_b(\R^d)$ with $[f]_\eta\leq 1$,
\begin{align*}
|\nabla P_{r,t}^{\mu^2}f(x)|&=\left|\int_{\R^d} \nabla p_{r,t}^{\mu^2}(\cdot, y)(x)f(y)\,\d y\right|\\
&=\left|\int_{\R^d} \nabla p_{r,t}^{\mu^2}(\cdot, y)(x)[f(y)-f(\theta_{r,t}(x))]\,\d y\right|\\
&\leq\int_{\R^d} |\nabla p_{r,t}^{\mu^2}(\cdot, y)(x)|\times|f(y)-f(\theta_{r,t}(x))|\,\d y\\
&\leq C(t-r)^{-1/\alpha}\int_{\R^d}p^\alpha_{t-r}\left(\theta_{r,t}(x)-y\right)|\theta_{r,t}(x)-y|^\eta\,\d y\\
&=C(t-r)^{-1/\alpha}\int_{\R^d}p^\alpha_{t-r}(z)|z|^\eta\,\d y\\
&=C(t-r)^{-1/\alpha}\E \big[ |L_{t-r}|^\eta\big]\\
&=C(t-r)^{(\eta-1)/\alpha}\E \big[ |L_{1}|^\eta\big].
\end{align*}
This implies \eqref{g0'} since $\E \big[ |L_{1}|^\eta\big]<\infty$.

\smallskip

\noindent
(ii) By \eqref{s-s}, it is not hard to get
\begin{align*}
&\left|\frac{|\mathrm{det}(\sigma_t^{-1}(\cdot,\mu^1_t))|}{|\sigma_t^{-1}(\cdot,\mu^1_t)y|^{d+\alpha}} -\frac{|\mathrm{det}(\sigma_t^{-1}(\cdot,\mu^2_t))|}{|\sigma_t^{-1}(\cdot,\mu^2_t)y|^{d+\alpha}}\right|\\
&=\left|\frac{|\sigma_t^{-1}(\cdot,\mu^2_t)y|^{d+\alpha}|\mathrm{det} (\sigma_t^{-1}(\cdot,\mu^1_t))|-|\sigma_t^{-1}(\cdot,\mu^1_t) y|^{d+\alpha}|\mathrm{det}(\sigma_t^{-1}(\cdot,\mu^2_t))|}{|\sigma_t^{-1}(\cdot,\mu^1_t)y|^{d+\alpha} |\sigma_t^{-1}(\cdot,\mu^2_t)y|^{d+\alpha}} \right|\\
&\leq \left|\frac{|\sigma_t^{-1}(\cdot,\mu^2_t)y|^{d+\alpha}[|\mathrm{det} (\sigma_t^{-1}(\cdot,\mu^1_t))|-|\mathrm{det}(\sigma_t^{-1}(\cdot,\mu^2_t))|]}{|\sigma_t^{-1}(\cdot,\mu^1_t)y|^{d+\alpha} |\sigma_t^{-1}(\cdot,\mu^2_t)y|^{d+\alpha}} \right|\\
&\quad+\left|\frac{|\mathrm{det} (\sigma_t^{-1}(\cdot,\mu^2_t))|[|\sigma_t^{-1}(\cdot,\mu^2_t)y|^{d+\alpha}-|\sigma_t^{-1}(\cdot,\mu^1_t) y|^{d+\alpha}]}{|\sigma_t^{-1}(\cdot,\mu^1_t)y|^{d+\alpha} |\sigma_t^{-1}(\cdot,\mu^2_t)y|^{d+\alpha}} \right|\\
&\leq C \frac{\W_{\eta}(\mu^1_t,\mu^2_t)}{|y|^{d+\alpha}}.
\end{align*}
Since we can rewrite \eqref{L34} as a principal value (p.v.) integral:
\begin{align*}
\A^{\nu}_t f(\cdot)&=\frac{1}{2}\,\mathrm{p.v.}\int_{\mathbb{R}^{d}}\big[f(\cdot+\sigma_t(\cdot,\nu_t)y)+f(\cdot-\sigma_t(\cdot,\nu_t)y)-2f(\cdot)
\big]\,\Pi(\d y)\\
&=\frac{\alpha\Gamma(\frac{d+\alpha}{2})}
    {2^{3-\alpha}\pi^{d/2}\Gamma(1-\frac\alpha2)}\,\mathrm{p.v.}\int_{\mathbb{R}^{d}}\big[f(\cdot+y)+f(\cdot-y)-2f(\cdot)
\big]\,\frac{|\mathrm{det}(\sigma_t^{-1}(\cdot,\nu_t))|}{|\sigma_t^{-1}(\cdot,\nu_t)y|^{d+\alpha}}\,\d y,
\end{align*}
it holds that
\begin{align*}
&|(\A^{\mu^1}_r-\A^{\mu^2}_r)P^{\mu^2}_{r,t}f|\\
&=C\left|
\mathrm{p.v.}\int_{\mathbb{R}^{d}}\big[P^{\mu^2}_{r,t}f(\cdot+y)+P^{\mu^2}_{r,t}f(\cdot-y)-2P^{\mu^2}_{r,t}f(\cdot)
\big]\left[\frac{|\mathrm{det}(\sigma_t^{-1}(\cdot,\mu^1_t))|}{|\sigma_t^{-1}(\cdot,\mu^1_t)y|^{d+\alpha}} -\frac{|\mathrm{det}(\sigma_t^{-1}(\cdot,\mu^2_t))|}{|\sigma_t^{-1}(\cdot,\mu^2_t)y|^{d+\alpha}}
\right]\d y\right|\\
&\leq C\,\mathrm{p.v.}\int_{\mathbb{R}^{d}}\big|
P^{\mu^2}_{r,t}f(\cdot+y)+P^{\mu^2}_{r,t}f(\cdot-y)-2P^{\mu^2}_{r,t}f(\cdot)
\big|\left|
\frac{|\mathrm{det}(\sigma_t^{-1}(\cdot,\mu^1_t))|}{|\sigma_t^{-1}(\cdot,\mu^1_t)y|^{d+\alpha}} -\frac{|\mathrm{det}(\sigma_t^{-1}(\cdot,\mu^2_t))|}{|\sigma_t^{-1}(\cdot,\mu^2_t)y|^{d+\alpha}}
\right|\d y\\
&\leq C\W_{\eta}(\mu^1_t,\mu^2_t)\,\mathrm{p.v.}\int_{\mathbb{R}^{d}}\big|
P^{\mu^2}_{r,t}f(\cdot+y)+P^{\mu^2}_{r,t}f(\cdot-y)-2P^{\mu^2}_{r,t}f(\cdot)
\big|\frac{\d y}{|y|^{d+\alpha}}\\
&= C\W_{\eta}(\mu^1_t,\mu^2_t)|\scr D^\alpha P_{r,t}^{\mu^2}f|,
\end{align*}
where $\scr D^\alpha$ is a fractional derivative operator of order $\alpha$ (cf. \cite[(1.23)]{MZ}) with
$$|\scr D^\alpha f|(x):=\int_{\R^d}|f(x+y)+f(x-y)-2f(x)|\,\frac{\d y}{|y|^{d+\alpha}}.$$
By \cite[Theorem 1.1\,(ii)]{MZ}, for any $0\leq r <t\leq T$,
$$
    |\scr D^\alpha p_{r,t}^{\mu^2}(\cdot,z)|(x)\leq C(t-r)^{-1}p_{t-r}^\alpha\left(\theta_{r,t}(x)-z\right).
$$
Then we get for all $f\in\scr B_b(\R^d)$ with $[f]_\eta\leq 1$,
\begin{align*}
|\scr D^\alpha P_{r,t}^{\mu^2}f|(x)
&=\int_{\R^d}\left|\int_{\R^d}\big[p_{r,t}^{\mu^2}(x+y,z)+p_{r,t}^{\mu^2}(x-y,z)-2p_{r,t}^{\mu^2}(x,z)\big]f(z)\,
\d z\right|\frac{\d y}{|y|^{d+\alpha}}\\
&=\int_{\R^d}\left|\int_{\R^d}\big[p_{r,t}^{\mu^2}(x+y,z)+p_{r,t}^{\mu^2}(x-y,z)-2p_{r,t}^{\mu^2}(x,z)\big][f(z)-f(\theta_{r,t}(x))]\,
\d z\right|\frac{\d y}{|y|^{d+\alpha}}\\
&\leq\int_{\R^d}|\scr D^\alpha p_{r,t}^{\mu^2}(\cdot,z)|(x)|f(\theta_{r,t}(x))-f(z)|\,\d z\\
&\leq C(t-r)^{-1}\int_{\R^d}p_{t-r}^\alpha(\theta_{r,t}(x)-z)|\theta_{r,t}(x)-z|^\eta\,\d z\\
&=C(t-r)^{-1}\E\big[|L_{t-r}|^\eta\big]\\
&=C(t-r)^{-1+\frac\eta\alpha}\E\big[|L_{1}|^\eta\big].
\end{align*}
This together with the above estimate implies \eqref{g2'}.
\end{proof}

\begin{lem}\label{dunh4}
Let $0\leq s<t\leq T$, $\mu^1,\mu^2\in C([0,T];\scr P_\eta)$ for $\eta\in(0,\alpha)$ and $f\in\scr B_b(\R^d)$. Then
$$P_{s,t}^{\mu^1}f = P_{s,t}^{\mu^2}f + \int_s^t  P_{s,r}^{\mu^1}
\big\<b_r(\cdot, \mu_r^1)-b_r(\cdot, \mu_r^2),\nabla P_{r,t}^{\mu^2}f\big\>\,\d r+\int_s^tP_{s,r}^{\mu^1}
(\scr A_r^{\mu_r^1}-\scr A_r^{\mu_r^2})P_{r,t}^{\mu^2}f\,\d r.$$
If furthermore $b$ does not depend on distribution variable, then
$$P_{s,t}^{\mu^1}f = P_{s,t}^{\mu^2}f +\int_s^tP_{s,r}^{\mu^1}
(\scr A_r^{\mu_r^1}-\scr A_r^{\mu_r^2})P_{r,t}^{\mu^2}f\,\d r.$$
\end{lem}

\begin{proof}
By a standard approximation argument, it suffices to prove the desired assertion
for $f\in C_b^2(\R^d)$. By the backward Kolmogorov equation, see \cite[Theorem 1.1]{MZ}, it holds that
$$
\frac{\partial P_{r,t}^{\mu^2}f}{\partial r}
=-\<b_r(\cdot, \mu_r^2),\nabla P_{r,t}^{\mu^2}f\>-\A_r^{\mu^2}(P_{r,t}^{\mu^2}f),\quad 0\leq r< t\leq T,
$$
where $\A^{\nu}_r f$ is given by \eqref{L34}.
By It\^{o}'s formula, we have the forward Kolmogorov equation
$$\frac{\partial P_{s,r}^{\mu^1}f}{\partial r}
=P_{s,r}^{\mu^1} [\<b_r(\cdot, \mu_r^1),\nabla f\>+\A_r^{\mu^1}f],\quad  0\leq s< r\leq T.$$
Hence, we have
\begin{align*}
    P_{s,t}^{\mu^1}f -P_{s,t}^{\mu^2}f
    &=\int_s^t
    \frac{\partial}{\partial r}[P_{s,r}^{\mu^1}P^{\mu^2}_{r,t}f]\,\d r\\
    &=\int_s^t  P_{s,r}^{\mu^1}
\big\<b_r(\cdot, \mu_r^1)-b_r(\cdot, \mu_r^2),\nabla P_{r,t}^{\mu^2}f\big\> \d r+\int_s^tP_{s,r}^{\mu^1}
(\scr A_r^{\mu_r^1}-\scr A_r^{\mu_r^2})P_{r,t}^{\mu^2}f\,\d r,
\end{align*}
which implies the first assertion. Clearly, the second assertion follows immediately from the first one.
\end{proof}

\begin{lem}\label{vdh22s}
If $(A)$ holds, then for all $\gg\in \scr P_{\eta}$, $\mu^i\in C([0,T];\scr P_{\eta})$, $i=1,2$, and $\delta>0$,
$$
   \sup_{t\in[0,T]}\e^{-\delta t}\W_\eta\big(\L_{{X}_{t}^{\gg,\mu^1}},\L_{{X}_{t}^{\gg,\mu^2}}\big)
    \leq C\left(\delta^{\frac{1}{\alpha}-\frac{\eta}{\alpha}-1}+
    \delta^{-\frac{\eta}{\alpha}}\right)
    \sup_{t\in[0,T]}\e^{-\delta t}\W_{\eta}(\mu^1_t,\mu^2_t).
$$
If $(A')$ holds, then for all $\gg\in \scr P_{\eta}$, $\mu^i\in C([0,T];\scr P_{\eta})$, $i=1,2$, and $\delta>0$,
$$
    \sup_{t\in[0,T]}\e^{-\delta t}\W_{\eta}\big(\L_{{X}_{t}^{\gg,\mu^1}},\L_{{X}_{t}^{\gg,\mu^2}}\big)
    \leq C\delta^{-\frac{\eta}{\alpha}}\sup_{t\in[0,T]}\e^{-\delta t}\W_{\eta}(\mu^1_t,\mu^2_t).
$$
\end{lem}

\begin{proof}
Assume $(A)$. It follows form the definition of $\W_{\eta}$ and Lemma \ref{dunh4} that
\begin{align*}
    \W_\eta\big(\L_{{X}_{t}^{\gg,\mu^1}},\L_{{X}_{t}^{\gg,\mu^2}}\big)
    &=\sup_{f\in\scr B_b(\R^d),[f]_\eta\leq 1}\left|
    \int_{\R^d}\big[
    P_{t}^{\mu^1}f(x)-P_{t}^{\mu^2}f(x)\big]\,\gg(\d x)\right|\\
    &\leq \sup_{f\in\scr B_b(\R^d),[f]_\eta\leq 1}\left|\int_{\R^d}\gg(\d x) \int_0^t  P_{0,r}^{\mu^1}
\big\<b_r(\cdot, \mu_r^1)-b_r(\cdot, \mu_r^2),\nabla P_{r,t}^{\mu^2}f\big\>(x)\,\d r
    \right|\\
    &\quad+\sup_{f\in\scr B_b(\R^d),[f]_\eta\leq 1}\left|\int_{\R^d}\gamma(\d x)\int_0^tP_{0,r}^{\mu^1}
\{(\scr A_r^{\mu_r^1}-\scr A_r^{\mu_r^2})P_{r,t}^{\mu^2}f\}(x)\,\d r\right|\\
&=:\sum_{i=1}^2\sup_{f\in\scr B_b(\R^d),[f]_\eta\leq 1}\mathsf{J}_i.
\end{align*}
By \eqref{b-b} and \eqref{g0'}, we derive that
for all $f\in\scr B_b(\R^d)$ with $[f]_\eta\leq 1$,
\begin{align*}
    \big\|\big\<b_r(\cdot, \mu_r^1)-b_r(\cdot, \mu_r^2),\nabla P_{r,t}^{\mu^2}f(\cdot)\big\>\big\|_{\infty}
    &\leq\big\|b_r(\cdot, \mu_r^1)-b_r(\cdot, \mu_r^2)\big\|_{\infty}\big\|\nabla P_{r,t}^{\mu^2}f(\cdot)\big\|_{\infty}\\
    &\leq C(t-r)^{-\frac1\alpha+\frac\eta\alpha}\W_{\eta}(\mu^1_r,\mu^2_r).
\end{align*}
Then we get for all $t\in[0,T]$, $\delta>0$ and $f\in\scr B_b(\R^d)$ with $[f]_\eta\leq 1$,
\begin{align*}
    \mathsf{J}_1&\leq C\int_0^t(t-r)^{-\frac1\alpha+\frac\eta\alpha}\W_{\eta}(\mu^1_r,\mu^2_r)\,\d r\\
    &=C\e^{\delta t}\int_0^t\e^{-\delta r}\W_{\eta}(\mu^1_r,\mu^2_r)\cdot(t-r)^{-\frac1\alpha+\frac\eta\alpha}
    \e^{-\delta (t-r)}\,\d r\\
    &\leq C\e^{\delta t}\sup_{s\in[0,T]}\e^{-\delta s}\W_{\eta}(\mu^1_s,\mu^2_s)
    \times\int_0^t(t-r)^{-\frac1\alpha+\frac\eta\alpha}
    \e^{-\delta (t-r)}\,\d r\\
    &\leq C\delta^{\frac1\alpha-\frac\eta\alpha-1}\e^{\delta t}\sup_{s\in[0,T]}\e^{-\delta s}\W_{\eta}(\mu^1_s,\mu^2_s),
\end{align*}
where in the last inequality we have used the fact that for any $\epsilon\in (0,1)$,
\begin{equation}\label{expint}
    \sup_{t\in[0,T]}
    \int_0^t(t-r)^{-\epsilon}\e^{-\delta(t-r)}\,\d r
    \leq\int_0^\infty r^{-\epsilon}\e^{-\delta r}\,\d r
    =\Gamma\left(1-\epsilon\right)\delta^{\epsilon-1}.
\end{equation}
By \eqref{g2'} and \eqref{expint}, for all $t\in[0,T]$, $\delta>0$ and $f\in\scr B_b(\R^d)$ with $[f]_\eta\leq 1$,
\begin{align*}
    \mathsf{J}_2&\leq C\int_0^t(t-r)^{-1+\frac\eta\alpha}\W_{\eta}(\mu^1_r,\mu^2_r)\,\d r\\
    &=C\e^{\delta t}\int_0^t\e^{-\delta r}\W_{\eta}(\mu^1_r,\mu^2_r)\cdot(t-r)^{-1+\frac\eta\alpha}
    \e^{-\delta (t-r)}\,\d r\\
    &\leq C\e^{\delta t}\sup_{s\in[0,T]}\e^{-\delta s}\W_{\eta}(\mu^1_s,\mu^2_s)\times
    \int_0^t(t-r)^{-1+\frac\eta\alpha}
    \e^{-\delta (t-r)}\,\d r\\
    &\leq C\delta^{-\frac\eta\alpha}\e^{\delta t}\sup_{s\in[0,T]}\e^{-\delta s}\W_{\eta}(\mu^1_s,\mu^2_s).
\end{align*}
Combining the bounds for $\mathsf{J}_i$, $i=1,2$, we obtain that for all $\delta>0$
\begin{align*}
    \sup_{t\in[0,T]}\e^{-\delta t}\W_\eta\big(\L_{{X}_{t}^{\gg,\mu^1}},\L_{{X}_{t}^{\gg,\mu^2}}\big)
    &\leq\sum_{i=1}^2\sup_{t\in[0,T]}\sup_{f\in\scr B_b(\R^d),[f]_\eta\leq 1}\e^{-\delta t}\mathsf{J}_i\\
    &\leq C\left(\delta^{\frac1\alpha-\frac\eta\alpha-1}+\delta^{-\frac\eta\alpha}\right)
    \sup_{s\in[0,T]}\e^{-\delta s}\W_{\eta}(\mu^1_s,\mu^2_s).
\end{align*}
This yields the first assertion. One can prove the second assertion by
repeating the argument above (with $\mathsf{J}_1=0$).
\end{proof}

\section{Proof of Theorem \ref{EUS}}\label{pf}

\begin{proof}[Proof of Theorem \ref{EUS}]
It follows from Lemma \ref{vdh22s} that for $\delta>0$ large enough, the map
$$
    \mu\mapsto \L_{X_{\cdot}^{\gg,\mu}}
$$
is strictly contractive in  $C([0,T];\scr P_\eta)$  under the complete metric
$$\sup_{t\in[0,T]}\e^{-\delta t}\W_{\eta}(\mu^1_t,\mu^2_t)$$
for $\mu^1,\mu^2\in C([0,T];\scr P_\eta)$.
Then it has a unique fixed point $\mu^\ast=\mu^\ast(\gg)\in  C([0,T];\scr P_\eta)$
such that $\mu^\ast=\L_{X_{\cdot}^{\gg,\mu^\ast}}$, and $X_t=X_{t}^{\gg,\mu^\ast}$ is the unique solution
to \eqref{E1} with $\L_{X_0}=\gamma\in\scr P_\eta$.

To prove the moment estimate, we will use a (random) time-change argument.
Let $S_t$ be an $\frac{\alpha}{2}$-stable
subordinator with the following Laplace transform:
$$
    \E\left[\e^{-rS_t}\right] = \e^{-2^{-1}t (2r)^{\alpha/2}},\quad r>0,\,t\geq 0,
$$
and let $W_t$ be a $d$-dimensional standard Brownian motion, which is independent of $S_t$.
The time-changed process $L_{t}:=W_{S_{t}}$ is a $d$-dimensional rotationally symmetric $\alpha$-stable L\'evy process
such that $\E\,\e^{\iup \<\xi, L_t\>}=\e^{-t|\xi|^\alpha/2}$ for $\xi\in\R^d$, see e.g.\ \cite{sato}.
Using the subordination representation, \eqref{E1} can be written in the following form
$$
    X_t=X_0+\int_0^tb_r(X_r,\L_{X_r})\,\d r+\int_0^t\sigma_r(X_r,\L_{X_r})\,\d W_{S_r},
$$
where $\L_{X_0}\in\scr P_\eta$. It is easy to see that $(A)$ or $(A')$ implies
for $x\in\R^d$ and $\gamma\in\scr P_\eta$,
$$
    \sup_{t\in[0,T]}|b_t(x,\gamma)|\leq C\big(1+|x|+\gamma(|\cdot|^\eta)\big).
$$
Since $\sigma$ is bounded due to \eqref{s-s}, we obtain for all $s\in[0,T]$,
\begin{align*}
    \E\left[\sup_{t\in[0,s]}|X_t|^\eta\right]
    &\leq C \E\left[|X_0|^\eta\right]+C\E\left[
    \int_0^s|b_r(X_r,\L_{X_r})|^\eta\,\d r
    \right]\\
    &\quad+C\E\left[
    \sup_{t\in[0,T]}\left|
    \int_0^t\sigma_r(X_r,\L_{X_r})\,\d W_{S_r}
    \right|^\eta
    \right]\\
    &\leq C \E\left[|X_0|^\eta\right]+C\int_0^s\left(1+
    \E\left[|X_r|^\eta\right]
    \right)\d r+C\E\left[S_T^{\eta/2}\right]\\
    &\leq C\left(1+\E\left[|X_0|^\eta\right]\right)
    +C\int_0^s\E\left[\sup_{t\in[0,r]}|X_t|^\eta\right]\d r,
\end{align*}
which together with the Gronwall inequality yields that
$$
    \E\left[\sup_{t\in[0,s]}|X_t|^\eta\right]\leq C\left(1+\E\left[|X_0|^\eta\right]\right)
    \e^{Cs}\leq C\left(1+\E\left[|X_0|^\eta\right]\right), \quad s\in[0,T].
$$
This completes the proof.
\end{proof}

\section{Appendix}\label{appd}

\begin{lem}\label{ss21dd}
    Let $\alpha\in(1/2,1)$ and $L_1$ be a \textup{(}stable\textup{)} random variable with $\E\,\e^{\i\xi L_1}=\e^{-|\xi|^\alpha}$.
    Then there exist $c>0$ and $\varrho>0$ such that
    \begin{equation}\label{eqjj}
        c=\alpha\E\left[\operatorname{sgn}(c+\varrho L_1)|c+\varrho L_1|^{1-\alpha}\right].
    \end{equation}
\end{lem}

\begin{proof}
    Let
    $$
        g(c,\varrho):=c-\alpha\E\left[\operatorname{sgn}(c+\varrho L_1)|c+\varrho L_1|^{1-\alpha}\right],
        \quad c>0,\varrho\geq0.
    $$
    It follows from the dominated convergence theorem that for $c>0$ and $\varrho\geq0$,
    \begin{align*}
        &\lim_{\epsilon_1\rightarrow0,\epsilon_2\rightarrow0}g(c+\epsilon_1,\varrho+\epsilon_2)\\
        &\qquad=c-\alpha\lim_{\epsilon_1\rightarrow0,\epsilon_2\rightarrow0}\E\left[
        \operatorname{sgn}(c+\varrho L_1+\epsilon_1+\epsilon_2L_1)\I_{\{c+\varrho L_1\neq0\}}
        |c+\varrho L_1+\epsilon_1+\epsilon_2L_1|^{1-\alpha}
        \right]\\
        &\qquad=c-\alpha\E\left[
        \operatorname{sgn}(c+\varrho L_1)\I_{\{c+\varrho L_1\neq0\}}
        |c+\varrho L_1|^{1-\alpha}
        \right]\\
        &\qquad=g(c,\varrho),
    \end{align*}
    which means that $g(c,\varrho)$ is continuous for $c>0$ and $\varrho\geq0$. Using the dominated convergence theorem again,
    we know that for $c>0$,
    $$
        \lim_{\varrho\downarrow0}g(c,\varrho)=c-\alpha c^{1-\alpha}=c^{1-\alpha}(c^\alpha-\alpha).
    $$
    Pick $0<c_1<c_2$ such that $c_1^\alpha<\alpha$ and $c_2^\alpha>\alpha$. Then we get
    \begin{align*}
        &\lim_{\varrho\downarrow0}g(c_1,\varrho)=c_1^{1-\alpha}(c_1^\alpha-\alpha)<0,\\
        &\lim_{\varrho\downarrow0}g(c_2,\varrho)=c_2^{1-\alpha}(c_2^\alpha-\alpha)>0.
    \end{align*}
    Now we conclude that there exist $c\in(c_1,c_2)$ and $\varrho>0$ such that
    $g(c,\varrho)=0$, and this completes the proof.
\end{proof}

\begin{exa}\label{conter}
    Let $d=1$ and $L_t$ be a symetric $\alpha$-stable process on $\R$ with
    $\E\,\e^{\i\xi L_t}=\e^{-t|\xi|^\alpha}$ \textup{(}$1/2<\alpha<1$\textup{)}. Let
    $$
        b(\gamma):=\int_\R\operatorname{sgn}(x)|x|^{1-\alpha}\,\gamma(\d x),\quad \gamma\in\scr P.
    $$
    It is easy to see that
    $$
        |b(\gamma)-b(\tilde{\gamma})|\leq2^\alpha\W_{1-\alpha}(\gamma,\tilde{\gamma})
        \leq2\W_{1-\alpha}(\gamma,\tilde{\gamma}),\quad \gamma,\tilde{\gamma}\in\scr P_{1-\alpha}.
    $$
    By Lemma \ref{ss21dd}, we can pick two constants $c>0$ and $\varrho>0$ such that \eqref{eqjj} holds.
    Consider the McKean-Vlasov SDE on $\R$:
    \begin{equation}\label{exasde}
        \d X_t=b(\L_{X_t})\,\d t+\varrho\,\d L_t.
    \end{equation}
    Since
    $$
        b(\L_{L_t})=\E\left[
        \operatorname{sgn}(L_t)|L_t|^{1-\alpha}
        \right]=0,
    $$
    we know that $X_t=L_t$ is a solution to \eqref{exasde} with $X_0=0$. Next, we will show that $X_t=ct^{1/\alpha}+\varrho L_t$ also
    solves \eqref{exasde}. To this aim, we use the scaling property of $L_t$ to get that for all $s\in(0,T]$,
    \begin{align*}
        b(\L_{cs^{1/\alpha}+\varrho L_s})
        &=\E\left[
        \operatorname{sgn}(cs^{1/\alpha}+\varrho s^{1/\alpha}L_1)|cs^{1/\alpha}+\varrho s^{1/\alpha}L_1|^{1-\alpha}
        \right]\\
        &=s^{\frac1\alpha-1}\E\left[
        \operatorname{sgn}(c+\varrho L_1)|c+\varrho L_1|^{1-\alpha}
        \right],
    \end{align*}
    which, together with \eqref{eqjj}, implies that
    \begin{align*}
        \int_0^tb(\L_{cs^{1/\alpha}+\varrho L_s})\,\d s
        &=\int_0^ts^{\frac1\alpha-1}\,\d s\times
        \E\left[
        \operatorname{sgn}(c+\varrho L_1)|c+\varrho L_1|^{1-\alpha}
        \right]\\
        &=\alpha t^{1/\alpha}\E\left[
        \operatorname{sgn}(c+\varrho L_1)|c+\varrho L_1|^{1-\alpha}
        \right]\\
        &=ct^{1/\alpha}.
    \end{align*}
    This means that $X_t=ct^{1/\alpha}+\varrho L_t$ is a solution to \eqref{exasde} with $X_0=0$. Thus, the SDE \eqref{exasde} with initial
    value $X_0=0$ has at least two strong solutions: $L_t$ and $ct^{1/\alpha}+\varrho L_t$, where $c>0$
    and $\varrho>0$ are two constants satisfying \eqref{eqjj}.
\end{exa}

\noindent
\textbf{Acknowledgement.} C.-S.\ Deng is supported by National Natural Science Foundation of China (12371149)
and Natural Science Foundation of Hubei Province of China (2022CFB129).
X.\ Huang is supported by National Key R\&D Program of China (No. 2022YFA1006000) and National Natural Science Foundation of China (12271398).


\begin{thebibliography}{99}


\bibitem{CF} P.-E.\ Chaudru de Raynal, N.\ Frikha, \emph{Well-posedness for some non-linear SDEs
and related PDE on the Wasserstein space}, J. Math. Pures Appl. 159 (2022), 1--167.


\bibitem{Chen04} M.-F.\ Chen, \emph{From Markov Chains to Non-Equilibrium Particle Systems}.
Word Scientific, Singapore 2004 (2nd edn).

\bibitem{CZZ21} Z.-Q.\ Chen, X.\ Zhang, G.\ Zhao, \emph{Supercritical SDEs driven by
multiplicative stable-like L\'{e}vy processes}, Trans. Amer. Math. Soc. 374 (2021), 7621--7655.


\bibitem{DH23}
C.-S.\ Deng, X.\ Huang, \emph{Well-posedness for McKean-Vlasov SDEs with distribution dependent
stable noises}, arXiv:2306.10970.





\bibitem{HW22} X.\ Huang, F.-Y.\ Wang,  \emph{Singular McKean-Vlasov (reflecting) SDEs with distribution dependent noise},
J. Math. Anal. Appl. 514 (2022), no. 1, Paper No. 126301, 21 pp.

\bibitem{JMW} B.\ Jourdain, S.\ M\'{e}l\'{e}ard, W. A.\ Woyczynski, \emph{Nonlinear SDEs driven by L\'{e}vy processes
    and related PDEs}, ALEA Lat. Am. J. Probab. Math. Stat. 4 (2008), 1--29.

\bibitem{MZ} S.\ Menozzi,  X.\ Zhang, \emph{Heat kernel of supercritical nonlocal operators with unbounded drifts},
J. \'{E}c. polytech. Math. 9 (2022), 537--579.


\bibitem{McKean} H.P.\ McKean\ Jr.,   \emph{A class of Markov processes associated with nonlinear parabolic equations},
Proc. Nat. Acad. Sci. U.S.A. 56 (1966), 1907--1911.


\bibitem{NVS} N.\ Frikha, V.\ Konakov, S.\ Menozzi, \emph{Well-posedness of some non-linear stable driven SDEs},
Discrete Contin. Dyn. Syst. 41 (2021), 849--898.

\bibitem{23R} P.\ Ren,  \emph{Singular McKean-Vlasov SDEs: Well-posedness, regularities and Wang's Harnack inequality},
Stochastic Process. Appl. 156 (2023), 291--311.

\bibitem{RZ} M.\ R\"ockner, X.\ Zhang, \emph{Well-posedness of distribution dependent SDEs with singular drifts},
Bernoulli 27 (2021), 1131--1158.

\bibitem{sato}
K.\ Sato, \emph{L\'evy processes and Infinitely Divisible Distributions}.
Cambridge University Press, Cambridge 1999.



\bibitem{SZ} A.-S. Sznitman,   \emph{Topics in propagation of chaos}, in: Lect. Notes in Math., vol. 1464, Springer-Verlag, Berlin 1991.


\bibitem{W21a} F.-Y.\ Wang, \emph{Distribution dependent reflecting stochastic differential equations},
Sci. China Math. 66 (2023), 2411--2456.

\bibitem{HaW} M.\ Wu, Z.\ Hao, \emph{Well-posedness of density dependent SDE driven by $\alpha$-stable process with H\"{o}lder drifts}. Stochastic Process. Appl. 164 (2023) 416--442.


\bibitem{ZG} G. Zhao, \emph{On distribution dependent SDEs with singular drifts}, arXiv:2003.04829v3.

\end{thebibliography}
\end{document}